\documentclass[12pt,oneside]{amsart}
\usepackage[english]{babel} 
\usepackage{amsmath} 
\usepackage{amsthm}
\usepackage{amssymb}
\usepackage{geometry}
\usepackage{nicefrac} 
\usepackage{hyperref}
\usepackage{graphicx}
\usepackage{braket}
\usepackage{subfig}
\usepackage{csquotes}
\usepackage{color}

\newcommand{\I}{\mathcal{I}}

\newcommand{\R}{\mathbb{R}}

\numberwithin{equation}{section}

\theoremstyle{plain}
\newtheorem{theorem}{Theorem}[section]

\theoremstyle{plain}

\theoremstyle{plain}
\newtheorem{lemma}[theorem]{Lemma}

\theoremstyle{plain}

\theoremstyle{plain}
\newtheorem{proposition}[theorem]{Proposition}

\theoremstyle{remark}

\title[]{Nonlocal degenerate Isaacs operators: H\"older regularity}
\date{}
\author{ Isabeau Birindelli, Giulio Galise and Yannick Sire}
\address{Dipartimento di Matematica Guido Castelnuovo, Sapienza
Universit\`a di Roma, Piazzale Aldo Moro 5, 00185, Roma, Italy.}
\email[I. Birindelli]{isabeau@mat.uniroma1.it}
\email[G. Galise]{galise@mat.uniroma1.it}
\address{Johns Hopkins University, Krieger Hall, 3400 N. Charles Street, Baltimore, MD, 21231, USA}
\email{ysire1@jhu.edu}

\begin{document}

 \begin{abstract} We show that bounded viscosity solutions of some nonlocal degenerate Isaacs type operators of order $2s$ are H\"older continuous, provided $s$ is sufficiently close to 1. As an application we obtain a Liouville theorem.\end{abstract}

\maketitle

\section{Introduction and main result}
The last decades have seen a huge interest in nonlocal operators and the properties of solutions of nonlocal equations. Particularly about fully nonlinear equations with a natural ellipticity condition, Caffarelli and Silvestre investigated the regularity of solutions to the Dirichlet problem, developing suitable  ABP estimates, Krylov-Safonov theorem and Evans-Krylov one. These seminal works generated a tremendous amount of literature since these operators appear naturally in many physical models (see e.g. \cite{CS1,CS3,CS2}). 

On the other hand, in the recent years there has been several works dedicated to some degenerate second-order investigated thoroughly by the first two authors with Ishii and Leoni (see e.g \cite{BGI1,BGI2,BGI3,BGL}). These operators have also a geometric meaning  and a optimal control interpretation (see e.g. \cite{HL}-\cite{BR}). These operators are very degenerate in a strong sense and the regularity properties of their viscosity solutions is understood only in some cases.

In this paper we consider a nonlocal equation associated to {\sl pure} L\'evy processes but very degenerate in a similar way as the second-order case is. In the probability community, the case of {\sl linear} degenerate operators has been considered for example by Bass and Chen \cite{bassChen} where they prove that a solution does not support  any Harnack inequality but still is H\"older continuous. In this case the singularity/degeneracy is played by some singular L\'evy measure. Typically, one can consider the sum $\sum_{k} (-\partial^2_{k})^{s}$ for $s \in (0,1)$ where the L\'evy measure is the sum of Dirac masses supported on each axes. Such linear equations have been investigated in e.g. \cite{ROS1, MR3436398,DROSV}. 

In the present contribution, although the operator considered is defined through 
one dimensional fractional Laplacians computed along directions of $\R^N$, it catches two extremal directions by means of the structure \lq\lq$\inf+\sup$\rq\rq. Hence it is a fully nonlinear operator and its local counterpart
is degenerate elliptic. For operators that do not include both the minimal
and maximal direction, the issue of H\"older regularity of solutions is in general open,
both in the local and nonlocal setting, and only partial results are known (see e.g. \cite{BGI1,BGS,FG,FV,V}). 
Here, however,  since the operator sees extremal directions, at least for an exponent close enough to second order, one recovers H\"older regularity. 

\smallskip 

More concretely, given a direction $\xi\in\mathbb S^{N-1}$ and  $u:\R^N\mapsto\R$, a bounded smooth function, let $\I_\xi u$ be the integral operator 
\begin{equation}\label{eq1}
\I_\xi u(x)=C_s\int\limits_0^{+\infty}\frac{u(x+\tau\xi)+u(x-\tau\xi)-2u(x)}{\tau^{1+2s}}\,d\tau,
\end{equation}
where $s\in(0,1)$ and $C_s$ is a normalizing constant
such that 
$$
\lim_{s\to1^-}\I_\xi u(x)=\left\langle D^2u(x)\xi,\xi\right\rangle.
$$ 
 Our goal here is to consider a fully nonlinear, {\sl strongly degenerate}, Isaacs type operator, given by 
\begin{equation}\label{eq2}
\I u=\inf_{\xi\in\mathbb S^{N-1}}\I_\xi u+\sup_{\xi\in\mathbb S^{N-1}}\I_\xi u .
\end{equation}
The operator $\I u$ is a possible nonlocal counterpart of the sum $$\lambda_1(D^2u)+\lambda_N(D^2u),$$ 
corresponding to the smallest and the largest eigenvalue of the Hessian of $u$.\\
Other nonlinear integral operators, named  fractional truncated Laplacians, have been introduced  in \cite{BGT}. There the authors, starting from the Courant-Fischer  characterization of the eigenvalues, analyzed different nonlocal versions of partial sums of eigenvalues of the Hessian matrix.


Our aim here is to study the H\"older regularity of viscosity solutions of the equation
\begin{equation}\label{eq3}
\I u=f(x)\quad\text{in $\Omega$},
\end{equation}
where $f\in C(\Omega)$ and $\Omega$ is a domain of $\R^N$ (not necessarily bounded).

\smallskip


Linear equations associated with L\'evy processes have been considered by Ros-Oton and Serra in \cite{ROS1}. The problem we consider here, which is very degenerate and nonlinear, does not fall into the class investigated in \cite{ROS1} or the fully nonlinear nonlocal equations with rough kernels considered in \cite{serra}. The authors of \cite{ROS1,serra} also consider some degenerate processes, where the degeneracy stems from the singularity of the L\'evy spectral measure, i.e the L\'evy measure is not necessarily supported on a set of positive measure; despite the singularity, they still  recover (at least) H\"older regularity for {\sl all} $s\in (0,1)$, i.e. the whole range of possible orders. 

Surprisingly enough, in our case where the L\'evy process sees only few, but {\sl extremal}, directions on the sphere, one gets H\"older regularity provided the order of the equation is close enough to $2$. This is the object of the our main theorem. 
\begin{theorem}\label{th1}
There exists $s_0\in(\frac12,1)$ such that if $s\in(s_0,1)$ and  $u$ is a bounded viscosity solution of \eqref{eq3}, then $u\in C_{\text{loc}}^{0,\alpha}(\Omega)$ for some $\alpha=\alpha(s)\in(0,1)$.\\
Moreover, given $\omega\subset\subset\omega'\subset\subset\Omega$,  one has
\begin{equation}\label{HolderEstimate}
\left\|u\right\|_{C^{0,\alpha}(\omega)}\leq C
\end{equation}
where $C=C(s,\text{dist}(\omega,\partial\omega'),\left\|u\right\|_{L^\infty(\R^N)},\left\|f\right\|_{L^\infty(\omega')})$ is a positive constant.
\end{theorem}

Our approach to regularity is based on viscosity techniques, so non-divergence in nature and based on maximum principle and barriers \`a la Ishii-Lions (see e.g. \cite{IL} and \cite{BCI} respectively for local and nonlocal strictly elliptic equations).

 The one by Ros-Oton and Serra is of more variational nature based on De Giorgi's version of De Giorgi-Nash-Moser theorem. A compactness argument allows to get variable kernels; the fully nonlinear version of this argument was implemented in \cite{serra}. 

\medskip

As far as the regularity  up to the boundary is concerned, we consider solutions to the problem

\begin{equation}\label{eqDir2}
\begin{cases}
\I u=f(x) & \text{in $\Omega$}\\
u=0 & \text{in $\R^N\backslash\Omega$,}
\end{cases}
\end{equation}
where $f\in C(\Omega)\cap L^\infty(\Omega)$ and $\Omega$ is bounded.

\begin{theorem}\label{th3}
Let $\Omega$ be a $C^2$-bounded domain and let $s_0\in(\frac12,1)$ be the lower bound introduced in Theorem \ref{th1}. If $u$ is a viscosity solution of \eqref{eqDir2} with  $s>s_0$, then $u\in C^{0,\alpha}(\overline\Omega)$ for some $\alpha=\alpha(s)\in(0,1)$.
\end{theorem}

Let us remark that the proofs of Theorems \ref{th1} and \ref{th3} can be adapted to a larger class of operators that include \eqref{eq2}. For instance one can consider for $a,b>0$ and integers $1\leq k,m\leq N$:
$$
J_{k,m}u=a \I_k^-u+b \I_m^+u
$$
where $\I_k^-u$ and $\I_m^+u$ are respectively the $k$-minimal and the $m$-maximal nonlocal truncated Laplacians, see \cite{BGT}. For sake of clarity, we made the choice of treating the case $a=b=k=m=1$ i.e. \eqref{eq2}.

\medskip

Building on Theorem \ref{th1}
, one obtains the following Liouville theorem. 

\begin{theorem}\label{th2Intro}
Let $s_0$ be defined in Theorem \ref{th1}. 
If $s>s_0$ and $u\in C(\mathbb R^N)$ is a bounded viscosity solution of \eqref{eq3} with $f \equiv 0$ and $\Omega\equiv\R^N$, then $u$ is constant. 
\end{theorem}

Compactness results are an immediate consequence of Theorem \ref{th3}. So it would be possible, using Krein-Rutman theorem, to prove, for the exterior Dirichlet problem for $\I$, the existence of the principal eigenvalue \`a la Berestycki-Nirenberg-Varadhan \cite{BNV} and its related eigenfunction. 

\bigskip
\noindent
\textbf{Notations.} $\mathbb S^{N-1}$ denotes the unit sphere in $\R^N$. The ball of radius $r>0$ centered at $x\in \R^N$ is denoted by $B_r(x)$.\\
The scalar product between $x$ and $y$ in $\R^N$ is $x\cdot y$ and $|x|=\sqrt{x\cdot x}$.\\
The positive part of $a\in\R$ is $a_+=\max\left\{a,0\right\}$.\\
Given $u:\R^N\mapsto\R$, we shall adopt the standard notation
\begin{equation}\label{not}
\delta(u,x,y)=u(x+y)+u(x-y)-2u(x)\qquad x,y\in\R^N.
\end{equation}

\section{Preliminaries}
In this section we will prove four technical lemmas that are needed in the proofs of Theorems \ref{th1} and \ref{th3}.

\begin{lemma}\label{lemma}
Let $c\in(1,2]$ and let $h_c:[\frac12,1]\mapsto\R$ be the function defined by the formula
\begin{equation}\label{hc}
h_c(s)=\int\limits_0^{c}\frac{\log|1-\tau^2|}{\tau^{1+s}}\,d\tau+\int\limits_c^{+\infty}\frac{\log(1+\tau)}{\tau^{1+s}}\,d\tau.
\end{equation}
Then there exists a unique $s_0=s_0(c)\in(\frac12,1)$ such that $h_c(s)>0$ for $s\in\left[\frac12,s_0\right)$, $h_c(s_0)=0$ and $h_c(s)<0$ for $s\in(s_0,1]$.
\end{lemma}
\begin{proof}
We first note that for $s\in[\frac12,1]$ one has
$$
\frac{\log|1-\tau^2|}{\tau^{1+s}}\in L^1(0,c)\quad\text{and}\quad \frac{|\log(1+\tau)|}{\tau^{1+s}}\in L^1(c,+\infty).
$$
Moreover 
\begin{equation*}
\begin{split}
\left|\frac{\partial}{\partial s}\frac{\log|1-\tau^2|}{\tau^{1+s}}\right|&=\frac{|\log|1-\tau^2|\log(\tau)|}{\tau^{1+s}}\\
&\leq|\log|1-\tau^2|\log(\tau)|\max\left\{\frac{1}{\tau^{2}},\frac{1}{\tau^{\frac32}}\right\}\in L^1(0,c)
\end{split}
\end{equation*} 
and
\begin{equation*}
\begin{split}
\left|\frac{\partial}{\partial s}\frac{\log(1+\tau)}{\tau^{1+s}}\right|&=\frac{|\log(1+\tau)\log(\tau)|}{\tau^{1+s}}\\
&\leq\frac{|\log(1+\tau)\log(\tau)|}{\tau^{\frac32}}\in L^1(c,+\infty).
\end{split}
\end{equation*} 
Hence $h_c\in C^1([\frac12,1])$ and 
$$
h_c'(s)=-\int\limits_0^{c}\frac{\log|1-\tau^2|\log(\tau)}{\tau^{1+s}}\,d\tau-\int\limits_c^{+\infty}\frac{\log(1+\tau)\log(\tau)}{\tau^{1+s}}.
$$
Integrating by parts, we obtain that 
\begin{equation}\label{eq4}
h_c'(s)=-\frac1sh_c(s)+g_c(s)
\end{equation}
where 
$$
g_c(s)=\frac1s\left(c^{-s}\log(c)\log(c-1)+2\int\limits_0^{c}\frac{\tau^{1-s}\log(\tau)}{1-\tau^2}\,d\tau-\int\limits_c^{+\infty}\frac{\log(\tau)}{(1+\tau)\tau^s}\,d\tau\right).
$$
Since $c\in(1,2]$, then $g_c(s)<0$ and 
$$
h_c'(s)<-\frac1sh_c(s)\quad\forall s\in\left[\frac12,1\right].
$$
Hence the map $s\mapsto sh_c(s)$ is decreasing and
\begin{equation}\label{eq5}
h_c(1)\leq sh_c(s)\leq\frac 12h_c\left(\frac12\right)\qquad\forall s\in\left[\frac12,1\right].
\end{equation}
Moreover, explicit computations yield
\begin{equation*}
h_c(1)=\frac{c-1}{c}\log(c-1)-\log(c)
\end{equation*}
and 
\begin{equation*}
\frac12h_c\left(\frac12\right)=\pi+\frac{\sqrt{c}-1}{\sqrt{c}}\log(\sqrt{c}-1)-\frac{\sqrt{c}+1}{\sqrt{c}}\log(\sqrt{c}+1).
\end{equation*}
Given that
$$
h_c(1)<0<h_c\left(\frac12\right),
$$
using \eqref{eq5} and the monotonicity of the map $s\mapsto sh_c(s)$, we conclude that there exists a unique $s_0=s_0(c)\in(\frac12,1)$ such that $h_c(s_0)=0$, $h_c(s)>0$ for $s\in\left[\frac12,s_0\right)$ and $h_c(s)<0$ for $s\in\left[s_0,1\right]$.
\end{proof}

If $\Omega\subset\mathbb R^N$ is domain, we denote by $d(x)$ the distance function from the complement of $\Omega$, that is
\begin{equation}\label{dist}
d(x):=
\begin{cases}
\text{dist}(x,\partial\Omega) & \text{if $x\in\Omega$}\\
0 & \text{if $x\in\mathbb R^N\backslash\Omega$.}
\end{cases}
\end{equation}

\begin{lemma}\label{lembar}
Let $\Omega$ be a $C^1$-bounded domain of $\mathbb R^N$ and suppose that $\left\{x_n\right\}_n\subset\Omega$ and $\left\{\xi_n\right\}_n\subset\mathbb S^{N-1}$ are respectively sequences such that $\displaystyle\lim_{n\to\infty}x_n=\bar x\in\partial\Omega$ and $\displaystyle\lim_{n\to\infty}\xi_n=\bar\xi\in\mathbb S^{N-1}$. Then for any $\tau\in\mathbb R$ it holds
\begin{equation}\label{eq1lemma}
\lim_{n\to\infty}\frac{d\left(x_n+d(x_n)\tau\xi_n\right)}{d(x_n)}=\left(1+\tau \nu(\bar x)\cdot\bar\xi\right)_+
\end{equation}
where $\nu(\bar x)$ denotes the unit inner normal to $\partial\Omega$ at $\bar x$.
\end{lemma}
\begin{proof}
Let 
$$
\tilde d(x):=
\begin{cases}
\text{dist}(x,\partial\Omega) & \text{if $x\in\Omega$}\\
-\text{dist}(x,\partial\Omega) & \text{if $x\in\mathbb R^N\backslash\Omega$}.
\end{cases}
$$
be the signed distance function from $\partial\Omega$. Note that 
\begin{equation}\label{5eq1}
d(x)=\tilde d_+(x)\quad\forall x\in\mathbb R^N
\end{equation}
and that, by the assumption on $\Omega$, there exists $\mu>0$ such that $\tilde d \in C^1(\Omega_\mu)$, where $\Omega_\mu=\left\{x\in\mathbb R^N:\;|\tilde d(x)|<\mu\right\}$. Moreover $\nabla \tilde d(x)=\nu(x)$ for any $x\in\partial\Omega$. \\
Let $\tau\in\mathbb R$. Since $x_n+d(x_n)\tau\xi_n\to\bar x$ as $n\to\infty$, we infer that, for $n$ sufficiently large, $x_n+d(x_n)\tau\xi_n\in B_\delta(\bar x)\subset\Omega_\delta$. Hence
$$
\tilde d(x_n+d(x_n)\tau\xi_n)=\tilde d(x_n)+\tilde d(x_n)\tau \nabla \tilde d(x_n)\cdot\xi_n+o(\tilde d(x_n))\quad \text{as $n\to\infty$}.
$$
Since $\tilde d(x)=d(x)$ for any $x\in\Omega$, then
\begin{equation*}
\frac{\tilde d(x_n+d(x_n)\tau\xi_n)}{d(x_n)}=1+\tau \nabla \tilde d(x_n)\cdot\xi_n+\frac{o( d(x_n))}{d(x_n)}\quad \text{as $n\to\infty$}
\end{equation*}
and 
 \begin{equation}\label{5eq2}
\lim_{n\to\infty}\frac{\tilde d(x_n+d(x_n)\tau\xi_n)}{d(x_n)}=1+\tau \nabla \tilde d(\bar x)\cdot\bar\xi=1+\tau \nu(\bar x)\cdot\bar\xi.
\end{equation}
 By \eqref{5eq1} and \eqref{5eq2}, the conclusion  \eqref{eq1lemma} follows.
\end{proof}

\begin{lemma}\label{lem1D}
For $s\in(0,1)$ and $\alpha\in(0,2s)$ let 
\begin{equation}\label{eql}
l(\alpha)=\int\limits_0^{+\infty}\frac{{(1+\tau)}^\alpha+{(1-\tau)}^\alpha_+-2}{\tau^{1+2s}}\,d\tau.
\end{equation}
Then
\begin{equation}\label{eq1lem1D}
l(\alpha)\begin{cases}
<0 & \text{if $\alpha<s$}\\
=0 & \text{if $\alpha=s$}\\
>0 & \text{if $\alpha>s$}\\
\end{cases}
\end{equation}
\end{lemma}
\begin{proof}
Arguing exactly as in \cite[Lemma 2.4]{BV}, which concerns the case $\alpha=s$, we obtain that for any $\alpha\in(0,2s)$
$$
l(\alpha)=\frac{\alpha}{2s}\int\limits_0^{+\infty}\frac{(1+\tau)^{\alpha-1}-(1+\tau)^{2s-\alpha-1}}{\tau^{2s}}\,d\tau.
$$
Then the conclusion easily follows.
\end{proof}

Next lemma can be deduced by the arguments used in \cite[Section 3]{BGT}. Nevertheless, for the reader's convenience, we give a self contained proof of it.

\begin{lemma}\label{lem0}
Let
$$u(x)=g\left(|x-y|^2\right)\qquad x\in\mathbb R^N,$$
where $y\in\mathbb R^N$ is fixed and $g:[0,+\infty)\mapsto\R$ is a concave function. Then for any $x\in\R^N$ it holds that
\begin{equation}\label{2eq1}
\delta(u,x,\tau\xi)\leq\delta(u,x,\tau\xi^\perp)\qquad\forall\tau\geq0
\end{equation}
whenever $\xi^\perp$ is a unit vector which is orthogonal to $x-y$.
\end{lemma}
\begin{proof}
For $x\in\R^N$ and $\xi\in\mathbb S^{N-1}$ we use the concavity inequality
$$
g(a+b)+g(a-b)\leq2g(a) \qquad\forall a\geq b\geq0
$$
with
$$
a=|x-y|^2+\tau^2\;,\quad b=2\tau\left|(x-y)\cdot\xi\right| 
$$ 
to infer that  for any $\tau\geq0$ one has
$$
g\left(|x-y+\tau\xi|^2\right)+g\left(|x-y-\tau\xi|^2\right)\leq 2g\left(|x-y|^2+\tau^2\right).
$$
Subtracting the quantity $2g(|x-y|^2)$ in both sides of the previous inequality we obtain that
\begin{equation}\label{2eq2}
\delta (u,x,\tau\xi)\leq 2g\left(|x-y|^2+\tau^2\right)-2g\left(|x-y|^2\right).
\end{equation}
On the other hand it is clear that 
\begin{equation}\label{2eq3}
2g\left(|x-y|^2+\tau^2\right)=u(x+\tau\xi^\perp)+u(x-\tau\xi^\perp)
\end{equation}
whenever $\xi^\perp\in\mathbb S^{N-1}$ is orthogonal to $x-y$. From \eqref{2eq2}-\eqref{2eq3} we conclude that \eqref{2eq1} holds.
\end{proof}

\section{Proof of the interior regularity}

Using the notation \eqref{not},
the operator \eqref{eq1} can be written as
$$
\I_\xi u(x)=C_s\int\limits_0^{+\infty}\frac{\delta(u,x,\tau\xi)}{\tau^{1+2s}}\,d\tau.
$$
Since the normalizing constant $C_s$ is not relevant for our arguments, henceforth  it will be omitted. 

\begin{proof}[Proof of Theorem \ref{th1}]
Let $\omega\subset\subset\omega'\subset\subset\Omega$ and pick $\varepsilon_0$ small enough such that 
$\omega_{2\varepsilon_0}=\left\{x\in\R^N:\,\text{dist}(x,\omega)<2\varepsilon_0\right\}\subset\subset\omega'$. For later purposes we can further assume that 
\begin{equation}\label{eq6}
\varepsilon_0\leq\frac{1}{\sqrt{1+2\left\|f\right\|_{L^\infty(\omega')}}}.
\end{equation}
Fix $z_0\in\omega$ and set 
$$
\phi(x,y)=u(x)-u(y)-M|x-y|^\alpha-L|x-z_0|^2,
$$
where 
\begin{equation}\label{eq7}
L=\frac{2\left\|u\right\|_{L^\infty(\R^N)}}{\varepsilon_0^2}.
\end{equation}
 We claim that for an appropriate choice of  $\alpha=\alpha(s)\in(0,1)$ and $M>0$, depending only on $\left\|u\right\|_{L^\infty(\R^N)}, \varepsilon_0,s$, then
\begin{equation*}
\sup_{(x,y)\in\R^N\times\R^N}\phi(x,y)\leq0,
\end{equation*}
which implies the desired result taking first $x=z_0$ and making $z_0$ vary in $\omega$. 

By contradiction we suppose that
\begin{equation}\label{eq8}
\sup_{(x,y)\in\R^N\times\R^N}\phi(x,y)>0.
\end{equation} By choosing $M$ sufficiently large, the supremum in \eqref{eq8} is in fact achieved at some point
$$
(\bar x,\bar y)\in\Delta_{z_0}:=\left\{(x,y)\in\Omega\times\Omega:\;|x-y|<\varepsilon_0,\;|x-z_0|<\varepsilon_0\right\}.
$$
Indeed, if $|x-z_0|\geq\varepsilon_0$ then for any $y\in\R^N$
$$
\phi(x,y)\leq2\left\|u\right\|_{L^\infty(\mathbb R^N)}-L\varepsilon_0^2=0,
$$
while if $|x-z_0|<\varepsilon_0$ and $|x-y|\geq\varepsilon_0$ one has
$$
\phi(x,y)\leq2\left\|u\right\|_{L^\infty(\mathbb R^N)}-M\varepsilon_0^\alpha\leq0,
$$
provided 
\begin{equation}\label{eq9}
M\geq\frac{2\left\|u\right\|_{L^\infty(\R^N)}}{\varepsilon_0^\alpha}.
\end{equation} 
Hence, by \eqref{eq8}, we have
\begin{equation}\label{eq10}
\max_{(x,y)\in\R^N\times\R^N}\phi(x,y)=\phi(\bar x,\bar y)>0,
\end{equation}
with 
\begin{equation}\label{eq11}
\bar x\neq\bar y\quad\text{and}\quad |\bar x-\bar y|<\varepsilon_0.
\end{equation}
Now we use the following two inequalities
\begin{equation*}
\begin{split}
\phi(x,\bar y)&\leq\phi(\bar x,\bar y)\quad\forall x\in\R^N\\
\phi(\bar x,y)&\leq\phi(\bar x,\bar y)\quad\forall y\in\R^N
\end{split}
\end{equation*}
to infer that for any $x\in\R^N$
\begin{equation}\label{eq12}
u(x)\leq\varphi_1(x):=u(\bar x)+M(|x-\bar y|^\alpha-|\bar x-\bar y|^\alpha)+L(|x-z_0|^2-|\bar x-z_0|^2) 
\end{equation}
 and that for any $y\in\R^N$
\begin{equation}\label{eq13}
u(y)\geq\varphi_2(y):=u(\bar y)+M|\bar x-\bar y|^\alpha-M|\bar x- y|^\alpha.
\end{equation}
Moreover $u(\bar x)=\varphi_1(\bar x)$ and $u(\bar y)=\varphi_2(\bar y)$.\\ Thus we can use $\varphi_1$ and $\varphi_2$ as test functions for $u$: if we let
$$
\psi_1(x)=
\begin{cases}
\varphi_1(x) & \text{if $|x-\bar x|\leq\frac12|\bar x-\bar y|$}\\
u(x) & \text{otherwise}
\end{cases}
$$
and
$$
\psi_2(y)=
\begin{cases}
\varphi_2(y) & \text{if $|y-\bar y|\leq\frac12|\bar x-\bar y|$}\\
u(x) & \text{otherwise},
\end{cases}
$$
then $\psi_1\in C^2\left(\overline{B_{\frac{|\bar x-\bar y|}{2}}(\bar x)}\right)$, $\psi_2\in C^2\left(\overline {B_{\frac{|\bar x-\bar y|}{2}}(\bar y)}\right)$  and 
\begin{equation}\label{eq14}
\begin{split}
f(\bar x)-f(\bar y)&\leq \I\psi_1(\bar x)-\I\psi_2(\bar y)\\&=\inf_{\xi\in\mathbb S^{N-1}}\I_\xi\psi_1(\bar x) -\sup_{\xi\in\mathbb S^{N-1}}\I_\xi\psi_2(\bar y)\\&\quad+\sup_{\xi\in\mathbb S^{N-1}}\I_\xi\psi_1(\bar x)-\inf_{\xi\in\mathbb S^{N-1}}\I_\xi\psi_2(\bar y).
\end{split}
\end{equation}
We now estimate, from above, the quantity $\inf_{\xi\in\mathbb S^{N-1}}\I_\xi\psi_1(\bar x) -\sup_{\xi\in\mathbb S^{N-1}}\I_\xi\psi_2(\bar y)$ by making the particular choice 
$$
\bar\xi=\frac{\bar x-\bar y}{|\bar x-\bar y|}.
$$
We have 
\begin{equation}\label{eq17}
\inf_{\xi\in\mathbb S^{N-1}}\I_\xi\psi_1(\bar x)\leq \I_{\bar\xi}\psi_1(\bar x)=\int\limits_0^{\frac{|\bar x-\bar y|}{2}}\frac{\delta(\varphi_1,\bar x,\tau\bar\xi)}{\tau^{1+2s}}\,d\tau+\int\limits_{\frac{|\bar x-\bar y|}{2}}^{+\infty}\frac{\delta(u,\bar x,\tau\bar\xi)}{\tau^{1+2s}}\,d\tau.
\end{equation}
By \eqref{eq12}, it holds that 
$$
\delta(u,\bar x,\tau\bar\xi)\leq\delta(\varphi_1,\bar x,\tau\bar\xi)\quad\forall\tau>0.
$$
Hence, for any $c\geq1/2$,
\begin{equation}\label{eq15}
\inf_{\xi\in\mathbb S^{N-1}}\I_\xi\psi_1(\bar x)\leq \I_{\bar\xi}\psi_1(\bar x)\leq\int\limits_0^{c{|\bar x-\bar y|}{}}\frac{\delta(\varphi_1,\bar x,\tau\bar\xi)}{\tau^{1+2s}}\,d\tau+\int\limits_{c{|\bar x-\bar y|}}^{+\infty}\frac{\delta(u,\bar x,\tau\bar\xi)}{\tau^{1+2s}}\,d\tau.
\end{equation}
Moreover, for any $\tau>0$,
\begin{equation}\label{eq16}
\delta(\varphi_1,\bar x,\tau\bar\xi)=M|\bar x-\bar y|^\alpha\left(\left|1+\frac{\tau}{|\bar x-\bar y|}\right|^\alpha+\left|1-\frac{\tau}{|\bar x-\bar y|}\right|^\alpha-2\right)+2L\tau^2.
\end{equation}
Combining \eqref{eq17}-\eqref{eq15}-\eqref{eq16} we get
 \begin{equation}\label{eq18}
\begin{split}
\inf_{\xi\in\mathbb S^{N-1}}\I_\xi\psi_1(\bar x)&\leq\frac{L}{1-s}(c|\bar x-\bar y|)^{2-2s}\\&\quad+M|\bar x-\bar y|^{\alpha-2s}\int\limits_0^c\frac{|1+\tau|^\alpha+|1-\tau|^\alpha-2}{\tau^{1+2s}}\,d\tau\\
&\quad+\int\limits_{c{|\bar x-\bar y|}}^{+\infty}\frac{\delta(u,\bar x,\tau\bar\xi)}{\tau^{1+2s}}\,d\tau.
\end{split}
\end{equation}
Similarly we have 
\begin{equation}\label{eq19}
\begin{split}
\sup_{\xi\in\mathbb S^{N-1}}\I_\xi\psi_2(\bar y)\geq\I_{\bar\xi}\psi_2(\bar y)&\geq-M|\bar x-\bar y|^{\alpha-2s}\int\limits_0^c\frac{|1+\tau|^\alpha+|1-\tau|^\alpha-2}{\tau^{1+2s}}\,d\tau\\
&\quad+\int\limits_{c{|\bar x-\bar y|}}^{+\infty}\frac{\delta(u,\bar y,\tau\bar\xi)}{\tau^{1+2s}}\,d\tau.
\end{split}
\end{equation}
Now if we subtract \eqref{eq18} and \eqref{eq19} we get

\begin{equation}\label{eq19'}
\begin{split}
\inf_{\xi\in\mathbb S^{N-1}}\I_\xi\psi_1(\bar x)&-\sup_{\xi\in\mathbb S^{N-1}}\I_\xi\psi_2(\bar y)\leq\frac{L}{1-s}(c|\bar x-\bar y|)^{2-2s}\\&\quad+2M|\bar x-\bar y|^{\alpha-2s}\int\limits_0^c\frac{|1+\tau|^\alpha+|1-\tau|^\alpha-2}{\tau^{1+2s}}\,d\tau\\
&\quad+\int\limits_{c{|\bar x-\bar y|}}^{+\infty}\frac{\delta(u,\bar x,\tau\bar\xi)-\delta(u,\bar y,\tau\bar\xi)}{\tau^{1+2s}}\,d\tau.
\end{split}
\end{equation}
Now we use the inequality
\begin{equation*}
\phi(\bar x\pm\tau\bar\xi,\bar y\pm\tau\bar\xi)\leq \phi(\bar x,\bar y)\quad \forall \tau>0,
\end{equation*}
which is a consequence of \eqref{eq10}, to infer that 
\begin{equation*}
\delta(u,\bar x,\tau\bar\xi)-\delta(u,\bar y,\tau\bar\xi)\leq2L\tau^2.
\end{equation*}
Hence
\begin{equation}\label{eq20}
\begin{split}
\int\limits_{c{|\bar x-\bar y|}}^{+\infty}\frac{\delta(u,\bar x,\tau\bar\xi)-\delta(u,\bar y,\tau\bar\xi)}{\tau^{1+2s}}&\,d\tau\leq2L\int_{c{|\bar x-\bar y|}}^1\tau^{1-2s}\,d\tau\\&\quad+\int_{1}^{+\infty}\frac{\delta(u,\bar x,\tau\bar\xi)-\delta(u,\bar y,\tau\bar\xi)}{\tau^{1+2s}}\,d\tau\\
&\leq\frac{L}{1-s}\left(1-(c|\bar x-\bar y|)^{2-2s}\right)+\frac4s\left\|u\right\|_{L^\infty(\R^N)}.
\end{split}
\end{equation}
The inequalities \eqref{eq19'}-\eqref{eq20} yield
\begin{equation}\label{eq21}
\begin{split}
\inf_{\xi\in\mathbb S^{N-1}}\I_\xi\psi_1(\bar x)&-\sup_{\xi\in\mathbb S^{N-1}}\I_\xi\psi_2(\bar y)\leq\frac{L}{1-s}+\frac4s\left\|u\right\|_{L^\infty(\R^N)}\\&\quad+2M|\bar x-\bar y|^{\alpha-2s}\int\limits_0^c\frac{|1+\tau|^\alpha+|1-\tau|^\alpha-2}{\tau^{1+2s}}\,d\tau.
\end{split}
\end{equation}
Now we estimate, from above, the quantity $\sup_{\xi\in\mathbb S^{N-1}}\I_\xi\psi_1(\bar x) -\inf_{\xi\in\mathbb S^{N-1}}\I_\xi\psi_2(\bar y)$ as follows:
\begin{equation*}
\begin{split}
\sup_{\xi\in\mathbb S^{N-1}}\I_\xi\psi_1(\bar x)&=\sup_{\xi\in\mathbb S^{N-1}}\int\limits_0^{+\infty}\frac{\delta(\psi_1,\bar x,\tau\xi)}{\tau^{1+2s}}\,d\tau\\
&\leq\sup_{\xi\in\mathbb S^{N-1}}\int\limits_0^{1}\frac{\delta(\psi_1,\bar x,\tau\xi)}{\tau^{1+2s}}\,d\tau+\sup_{\xi\in\mathbb S^{N-1}}\int\limits_{1}^{+\infty}\frac{\delta(\psi_1,\bar x,\tau\xi)}{\tau^{1+2s}}\,d\tau\\
&\leq\sup_{\xi\in\mathbb S^{N-1}}\int\limits_0^{1}\frac{\delta(\psi_1,\bar x,\tau\xi)}{\tau^{1+2s}}\,d\tau+\frac2s\left\|u\right\|_{L^\infty(\R^N)}\\
&\leq M\sup_{\xi\in\mathbb S^{N-1}}\int\limits_0^{1}\frac{\delta(\Gamma,\bar x,\tau\xi)}{\tau^{1+2s}}\,d\tau+\frac{L}{1-s}+\frac2s\left\|u\right\|_{L^\infty(\R^N)},
\end{split}
\end{equation*}
where $\Gamma(x)=|x-\bar y|^\alpha$. From  Lemma \ref{lem0}, applied to the function $\Gamma$ with $g(r)=r^\frac\alpha2$,  we also know that 
$$
\delta(\Gamma,\bar x,\tau\xi)\leq\delta(\Gamma,\bar x,\tau\xi^\perp) \quad\forall\tau>0,
$$
where $\xi^\perp$ is any unit vector which is orthogonal to $\bar x-\bar y$. Hence
\begin{equation}\label{eq22}
\begin{split}
\sup_{\xi\in\mathbb S^{N-1}}\I_\xi\psi_1(\bar x)
&\leq\frac2s\left\|u\right\|_{L^\infty(\R^N)}+\frac{L}{1-s}+M\int\limits_0^{1}\frac{\delta(\Gamma,\bar x,\tau\xi^\perp)}{\tau^{1+2s}}\,d\tau\\
&=\frac2s\left\|u\right\|_{L^\infty(\R^N)}+\frac{L}{1-s}+2M|\bar x-\bar y|^{\alpha-2s}\int\limits_0^{\frac{1}{|\bar x-\bar y|}}\frac{\left(1+\tau^2\right)^{\frac\alpha2}-1}{\tau^{1+2s}}\,d\tau\\
&\leq \frac2s\left\|u\right\|_{L^\infty(\R^N)}+\frac{L}{1-s}+2M|\bar x-\bar y|^{\alpha-2s}\int\limits_0^{+\infty}\frac{\left(1+\tau^2\right)^{\frac\alpha2}-1}{\tau^{1+2s}}\,d\tau.
\end{split}
\end{equation}
In a similar way we also obtain
\begin{equation}\label{eq23}
\inf_{\xi\in\mathbb S^{N-1}}\I_\xi\psi_2(\bar y)\geq-\frac2s\left\|u\right\|_{L^\infty(\R^N)}-2M|\bar x-\bar y|^{\alpha-2s}\int\limits_0^{+\infty}\frac{\left(1+\tau^2\right)^{\frac\alpha2}-1}{\tau^{1+2s}}\,d\tau.
\end{equation}
Putting together \eqref{eq22}-\eqref{eq23} we have
\begin{equation}\label{eq24}
\begin{split}
\sup_{\xi\in\mathbb S^{N-1}}\I_\xi\psi_1(\bar x)&-\inf_{\xi\in\mathbb S^{N-1}}\I_\xi\psi_2(\bar y)\leq\frac{L}{1-s}+\frac4s\left\|u\right\|_{L^\infty(\R^N)}\\&\quad+4M|\bar x-\bar y|^{\alpha-2s}\int\limits_0^{+\infty}\frac{\left(1+\tau^2\right)^{\frac\alpha2}-1}{\tau^{1+2s}}\,d\tau.
\end{split}
\end{equation}
Using \eqref{eq14},\eqref{eq21} and \eqref{eq24} in we arrive at 
\begin{equation}\label{eq25}
-2\left\|f\right\|_{L^\infty(\omega')}\leq\frac{2L}{1-s}+\frac8s\left\|u\right\|_{L^\infty(\R^N)}+2M|\bar x-\bar y|^{\alpha-2s}g_c(\alpha)
\end{equation}
where 
$$
g_c(\alpha)=\int\limits_0^c\frac{|1+\tau|^\alpha+|1-\tau|^\alpha-2}{\tau^{1+2s}}\,d\tau+2\int\limits_0^{+\infty}\frac{\left(1+\tau^2\right)^{\frac\alpha2}-1}{\tau^{1+2s}}\,d\tau.
$$
The goal is to select $c$ and $\alpha$ such that 
\begin{equation}\label{eq26}
g_c(\alpha)<0.
\end{equation}
Note that the latter integral in the definition of the function $g_c(\alpha)$ is always positive (and well defined) for any $\alpha\in(0,2s)$, while the first integral is negative, for any $c>0$, provided $\alpha\in(0,2s-1]$, see \cite[Lemma B.1]{barlesCGJ}-\cite[Lemma 3.6]{BGT}.\\
Once \eqref{eq26} is proved, using \eqref{eq6}-\eqref{eq7}-\eqref{eq11}-\eqref{eq25} we obtain
\begin{equation}\label{eq27}
\begin{split}
-2\left\|f\right\|_{L^\infty(\omega')}&\leq\frac{4\left\|u\right\|_{L^\infty(\R^N)}}{\varepsilon_0^2(1-s)}+\frac8s\left\|u\right\|_{L^\infty(\R^N)}+2M\varepsilon_0^{\alpha-2s}g_c(\alpha)\\
&\leq\frac{1}{\varepsilon_0^2}\left(4\frac{2-s}{s(1-s)}\left\|u\right\|_{L^\infty(\R^N)}+2M\varepsilon_0^{\alpha+2(1-s)}g_c(\alpha)\right).
\end{split}
\end{equation}
Choosing $M$ in such a way, in addition to the condition \eqref{eq9}, it holds that 
$$
4\frac{2-s}{s(1-s)}\left\|u\right\|_{L^\infty(\R^N)}+2M\varepsilon_0^{\alpha+2(1-s)}g_c(\alpha)\leq-1,
$$
then from \eqref{eq27} we arrive at
$$
-2\left\|f\right\|_{L^\infty(\omega')}\leq-\frac{1}{\varepsilon_0^2}
$$
which is a contradiction to the definition of $\varepsilon_0$, see \eqref{eq6}. 

\smallskip

It remains to prove \eqref{eq26}. For this we first note that
\begin{equation}\label{eq28}
g_c(1)=2\int\limits_0^{+\infty}\frac{\sqrt{1+\tau^2}-1}{\tau^{1+2s}}\,d\tau>0
\end{equation}
and that, by dominated convergence theorem, 
\begin{equation}\label{eq29}
g_c(0):=\lim_{\alpha\to0^+}g_c(\alpha)=0.
\end{equation}
Moreover $g_c'(\alpha)$ exists for any $\alpha\in(0,2s)$ and
\begin{equation}\label{eq30}
\begin{split}
g_c'(\alpha)&=\int\limits_0^c\frac{|1+\tau|^\alpha\log|1+\tau|+|1-\tau|^\alpha\log|1-\tau|}{\tau^{1+2s}}\,d\tau\\&\quad+\int\limits_0^{+\infty}\frac{(1+\tau^2)^\frac{\alpha}{2}\log(1+\tau^2)}{\tau^{1+2s}}\,d\tau.
\end{split}
\end{equation}
Indeed 
\begin{equation*}
\begin{split}
\left|\frac{\partial}{\partial \alpha}\frac{|1+\tau|^\alpha+|1-\tau|^\alpha-2}{\tau^{1+2s}}\right|&=\left|\frac{|1+\tau|^\alpha\log|1+\tau|+|1-\tau|^\alpha\log|1-\tau|}{\tau^{1+2s}}\right|\\&\leq \frac{K}{\tau^{2s-1}}\in L^1(0,c),
\end{split}
\end{equation*}
for a positive constant $K$ independent of $\alpha$, and if $\alpha<2s-\varepsilon$ for some $\varepsilon>0$,
 \begin{equation*}
\begin{split}
\left|\frac{\partial}{\partial \alpha}\frac{\left(1+\tau^2\right)^{\frac\alpha2}-1}{\tau^{1+2s}}\right|&=
\left|\frac{(1+\tau^2)^\frac{\alpha}{2}\log(1+\tau^2)}{\tau^{1+2s}}\right|\\&\leq\frac{\left(1+\tau^2\right)^{\frac{2s-\varepsilon}{2}}\log(1+\tau^2)}{\tau^{1+2s}}\in L^1(0,+\infty).
\end{split}
\end{equation*}
Hence, it is possible to differentiate $g_c(\alpha)$ under the integral sign to obtain \eqref{eq30}. Moreover, again by dominated convergence theorem,
\begin{equation}\label{eq31}
(g_c)'_+(0)=\lim_{\alpha\to0^+}g_c'(\alpha)=\int\limits_0^c\frac{\log|1-\tau^2|}{\tau^{1+2s}}\,d\tau+\int\limits_0^{+\infty}\frac{\log(1+\tau^2)}{\tau^{1+2s}}\,d\tau.
\end{equation}
Arguing as above, it is also possible to differentiate twice $g_c(\alpha)$ under the integral sign, to discover that for any $\alpha\in(0,2s)$ 
\begin{equation*}
\begin{split}
g_c''(\alpha)&=\int\limits_0^c\frac{|1+\tau|^\alpha\log^2|1+\tau|+|1-\tau|^\alpha\log^2|1-\tau|}{\tau^{1+2s}}\,d\tau\\&\quad+\frac12\int\limits_0^{+\infty}\frac{(1+\tau^2)^\frac{\alpha}{2}\log^2(1+\tau^2)}{\tau^{1+2s}}\,d\tau.
\end{split}
\end{equation*}
Since $g_c''(\alpha)>0$ for any $\alpha\in(0,2s)$, then $g_c(\alpha)$ is  strictly convex. Taking into account of \eqref{eq28}-\eqref{eq29}, then $g_c(\alpha)$ will be negative in a right neighborhood of $0$, provided $(g_c)_+'(0)<0$. Here is the point where we select the parameter $c$ in a suitable way. Using \eqref{eq31} we have 
\begin{equation*}
(g_c)_+'(0)=\frac12\left(\int\limits_0^{c^2}\frac{\log|1-\tau^2|}{\tau^{1+s}}\,d\tau+\int\limits_{c^2}^{+\infty}\frac{\log(1+\tau)}{\tau^{1+s}}\,d\tau\right)=\frac12 h_{c^2}(s),
\end{equation*}
where  $h_c(s)$ is the function defined by \eqref{hc}. In view of Lemma \ref{lemma}, for any positive $c$ such that $c^2\in(1,2])$, i.e. for any $c\in(1,\sqrt{2}]$, then 
$$
(g_c)_+'(0)<0
$$
and  $g_c(\alpha)<0$ for any $\alpha\in(0,\alpha^*)$, where $\alpha^*=\alpha^*(c,s)>0$ is such that $g_c(\alpha^*)=0$. 
\end{proof}

\section{Proof of the global regularity}

We start with the following proposition which will be used in Proposition \ref{probarr} to obtain estimates of  solutions of \eqref{eqDir2} in terms of powers of the distance function $d(x)$ defined in \eqref{dist} . The proof follows somehow the one given in \cite[Lemma 3.3]{QSX} for the nonlocal Pucci's operators.
\begin{proposition}\label{propdist}
Let $\Omega$ be a $C^2$-bounded domain and let $\alpha\in(0,s)$. Then there exists $\mu$ and $m$ positive constants such that 
\begin{equation}\label{eq1gr}
\I d^\alpha(x)\leq -md^{\alpha-2s}(x)
\end{equation}
for any $x\in \Omega_\mu=\left\{x\in\Omega:\;d(x)<\mu\right\}$.
\end{proposition}  
\begin{proof}
Suppose not, then we can find a sequence $\left\{x_n\right\}_n$ of points of $\Omega$ such that 
\begin{equation}\label{eq2gr}
\lim_{n\to\infty} x_n=\bar x\in\partial\Omega
\end{equation}
and
\begin{equation}\label{eq3gr}
\lim_{n\to\infty} d^{2s-\alpha}(x_n)\I d^\alpha(x_n)\geq0.
\end{equation} 
Let $\left\{\xi_n\right\}_n$ be a sequence of unit vectors such that 
$$
\sup_{\xi\in\mathbb S^{N-1}}\I_{\xi_n} d^\alpha(x_n)\leq \I_{\xi_n}d^\alpha(x_n)+\frac1n.
$$
Up to a subsequence we may suppose that $\xi_n\to\bar\xi\in\mathbb S^{N-1}$ as $n\to\infty$ and, from \eqref{eq3gr}, that
 \begin{equation}\label{eq4gr}
\lim_{n\to\infty} d^{2s-\alpha}(x_n)\left(\I_{\xi_n} d^\alpha(x_n)+\inf_{\xi\in\mathbb S^{N-1}}\I_{\xi} d^\alpha(x_n)\right)\geq0.
\end{equation} 
We first compute the limit 
$$
\lim_{n\to\infty} d^{2s-\alpha}(x_n)\I_{\xi_n} d^\alpha(x_n).
$$
With the change of variables $\tau=d(x_n)t$ we can write
$$
d^{2s-\alpha}(x_n)\I_{\xi_n} d^\alpha(x_n)=\int\limits_0^{+\infty}\frac{g_n(t)}{t^{1+2s}}\,dt
$$
where 
$$
g_n(t)={\left(\frac{d(x_n+d(x_n)t\xi_n)}{d(x_n)}\right)}^\alpha+{\left(\frac{d(x_n-d(x_n)t\xi_n)}{d(x_n)}\right)}^\alpha-2.
$$
The aim is to pass to the limit under the integral sign. Using Lemma \ref{lembar} and in particular \eqref{eq1lemma}, we infer that for $t$ small, say $t\in(0,\frac14)$ it holds that 
\begin{equation}\label{eq5gr}
\frac{d(x_n\pm d(x_n)t\xi_n)}{d(x_n)}\geq\frac12
\end{equation} 
provided $n$ is sufficiently large. Moreover, by the Lipschitz continuity of the distance function, we have for $t\in(0,\frac14)$ that 
\begin{equation}\label{eq6gr}
\frac{d(x_n\pm d(x_n)t\xi_n)}{d(x_n)}\leq1+t\leq \frac54
\end{equation} 
In view of \eqref{eq5gr}-\eqref{eq6gr} and using the fact that the  distance function $d(x)$ has bounded second derivatives for $x$ sufficiently close to $\partial\Omega$, we obtain that $g_n(t)$ is twice differentiable for $t\in(0,\frac14)$, in fact 
\begin{equation}\label{eq7gr}
|g_n(t)|\leq Ct^2
\end{equation}
for some positive constant $C$ independent on $n$.\\
On the other hand, again by the Lipschitz continuity of $d(x)$, we have
$$
\left|d^\alpha(x_n\pm d(x_n)t\xi_n)-d^\alpha(x_n)\right|\leq{\left|d(x_n\pm d(x_n)t\xi_n)-d(x_n)\right|}^\alpha\leq t^\alpha d^\alpha(x_n)
$$
and then, for $t\geq\frac14$ 
\begin{equation}\label{eq8gr}
|g_n(t)|\leq 2t^\alpha.
\end{equation}
Using \eqref{eq7gr}-\eqref{eq8gr} we infer that 
$$
\frac{|g_n(t)|}{t^{1+2s}}\leq\begin{cases}
C t^{1-2s} & \text{if $t\in(0,\frac14)$}\\
2t^{\alpha-1-2s} & \text{if $t\geq\frac14$}
\end{cases}\in L^1(0,+\infty).
$$
Since by \eqref{eq1lemma}
$$
\lim_{n\to\infty} g_n(t)={\left(1+t \nu(\bar x)\cdot\bar\xi\right)}^\alpha_++{\left(1-t \nu(\bar x)\cdot\bar\xi\right)}^\alpha_+-2
$$
pointwise in $(0,+\infty)$, we can use Lebesgue's dominated convergence theorem to obtain that
 \begin{equation}\label{eq9gr}
\begin{split}
\lim_{n\to\infty} d^{2s-\alpha}(x_n)\I_{\xi_n} d^\alpha(x_n)&=\int\limits_0^{+\infty}\frac{{\left(1+t \nu(\bar x)\cdot\bar\xi\right)}^\alpha_++{\left(1-t \nu(\bar x)\cdot\bar\xi\right)}^\alpha_+-2}{t^{1+2s}}\,dt\\
&=|\nu(\bar x)\cdot\bar\xi|^{2s}\int\limits_0^{+\infty}\frac{{\left(1+\tau\right)}^\alpha+{\left(1-\tau \right)}^\alpha_+-2}{\tau^{1+2s}}\,d\tau.
\end{split}
\end{equation}
We now estimate from above the limit
$$
\lim_{n\to\infty} d^{2s-\alpha}(x_n)\inf_{\xi\in\mathbb S^{N-1}}\I_{\xi} d^\alpha(x_n).
$$
By making the particular choice $\xi=\nu(\bar x)$, we have 
$$
\inf_{\xi\in\mathbb S^{N-1}}\I_{\xi} d^\alpha(x_n)\leq\I_{\nu(\bar x)}d^\alpha(x_n).
$$
Moreover, arguing as above, we also have 
\begin{equation}\label{eq10gr}
\lim_{n\to\infty} d^{2s-\alpha}(x_n)\I_{\nu(\bar x)} d^\alpha(x_n)=\int\limits_0^{+\infty}\frac{{\left(1+\tau\right)}^\alpha+{\left(1-\tau \right)}^\alpha_+-2}{\tau^{1+2s}}\,d\tau.
\end{equation}
From \eqref{eq9gr}-\eqref{eq10gr} and \eqref{eql} we infer that
\begin{equation*}
\lim_{n\to\infty} d^{2s-\alpha}(x_n)\left(\I_{\xi_n} d^\alpha(x_n)+\I_{\nu(\bar x)} d^\alpha(x_n)\right)=\left(|\nu(\bar x)\cdot\bar\xi|^{2s}+1\right)l(\alpha)<0,
\end{equation*}
since $\alpha<s$ (Lemma \ref{lem1D}). This contradicts \eqref{eq4gr} and the proof is then complete.
\end{proof}

\begin{proposition}\label{probarr}
Let $u\in C(\mathbb R^N)$ be a viscosity solution \eqref{eqDir2} with $f\in C(\Omega)\cap L^\infty(\Omega)$. Then for $\alpha<s$ there exists a positive constant $C$ such that 
\begin{equation}\label{globalestimate}
|u(x)|\leq Cd^\alpha(x)
\end{equation}
for any $x\in\Omega$.
\end{proposition}
\begin{proof}
Using Proposition \ref{propdist} we can pick $\mu>0$ small enough and $C>0$ sufficiently large such that the function $v(x)=Cd^\alpha(x)\in C^2(\Omega_\mu)$ satisfies 
\begin{equation}\label{eq1probarr}
v(x)\geq u(x)\quad\forall x\in\Omega\backslash\Omega_\mu
\end{equation} 
and 
\begin{equation}\label{eq2probarr}
\I v(x)\leq -\left(\left\|f\right\|_{L^\infty(\Omega)}+1\right)\quad\forall x\in\Omega_\mu.
\end{equation} 
If $u(x)>v(x)$ for some $x\in\Omega_\mu$ then, by \eqref{eq1probarr} and the fact that $v\equiv u$ in $\R^N\backslash\Omega$,  there would exist $x_0\in\Omega_\mu$ such that  
$$
0<u(x_0)-v(x_0)=\max_{\R^N}u-v.
$$
Hence, using $v$ as test function at $x_0$, we get
$$
\I v(x_0)\geq f(x_0)
$$
which is in contradiction to \eqref{eq2probarr}. This shows that $u(x)\leq Cd^\alpha(x)$ in $\Omega$. \\
The inequality $$u(x)\geq- Cd^\alpha(x)\quad\text{in $\Omega$}$$ can be obtained arguing as  above with the function $-u$ which is in turn a solution of \eqref{eqDir2} with $f$ replaced by $-f$.
\end{proof}

The proof of the interior regularity can be easily modified to yield global regularity of solutions of \eqref{eqDir2} once the boundary estimates \eqref{globalestimate} are available.

\begin{proof}[Sketch of the proof of Theorem \ref{th3}]
We adopt the same notations used in the proof of Theorem \ref{th1}. Let 
$$
\varepsilon_0=\frac{1}{\sqrt{1+2\left\|f\right\|_{L^\infty(\Omega)}}}
$$
and consider the function $\phi:\R^N\times\R^N\mapsto\R$ defined by
$$
\phi(x,y)=u(x)-u(y)-M|x-y|^\alpha.
$$
The theorem is proved once one shows that  
$$
\sup_{(x,y)\in\R^N\times\R^N}\phi(x,y)\leq0
$$
for an appropriate choice of $\alpha$ and $M$.\\
By contradiction we assume the contrary, that is
\begin{equation}\label{eq1th3}
\sup_{(x,y)\in\R^N\times\R^N}\phi(x,y)>0.
\end{equation}
Now we observe that if $M$ is sufficiently large then there exists $(\bar x,\bar y)\in\Omega\times\Omega$, with $\bar x\neq\bar y$, such that 
\begin{equation}\label{eq2th3}
|\bar x-\bar y|<\varepsilon_0\quad\text{and}\quad\max_{(x,y)\in\R^N\times\R^N}\phi(x,y)=\phi(\bar x,\bar y).
\end{equation}
If \eqref{eq2th3} is true, then one can follows line by line the proof of the interior regularity, starting from \eqref{eq12} and with $L=0$, in order to obtain the desired contradiction.\\
To prove \eqref{eq2th3} we first observe that  
\begin{equation}\label{eq3th3}
\phi(x,y)\leq0 \quad \forall x,y\in\R^N\;\;\text{s.t.}\;\;|x-y|\geq\varepsilon_0
\end{equation}
provided $M\geq\frac{2\left\|u\right|_{L^\infty}(\R^N)}{\varepsilon_0^\alpha}$. If instead $x\not\in\Omega$ then, using \eqref{globalestimate}, it holds that for every $y\in\R^N$ 
$$
u(x)-u(y)=-u(y)\leq Cd^\alpha(y)\leq C|x-y|^\alpha.
$$
Hence
\begin{equation}\label{eq4th3}
\phi(x,y)\leq0 \quad \forall (x,y)\in\R^N\backslash\Omega\times\R^N
\end{equation}
provided $M\geq C$. In a similar way, using \eqref{globalestimate},  
\begin{equation}\label{eq5th3}
\phi(x,y)\leq0 \quad \forall (x,y)\in\Omega\times\R^N\backslash\Omega.
\end{equation}
From \eqref{eq3th3}-\eqref{eq4th3}-\eqref{eq5th3} we infer that \eqref{eq2th3} holds and this conclude the proof.
\end{proof}

\section{Application: Liouville property}

In \cite[Section 7]{FV} the following  unilateral Liouville property was proved: any viscosity supersolution $u$, bounded from below, of the equation
\begin{equation}\label{Leq1}
\lambda_1(D^2u)+\lambda_N(D^2u)=0\qquad\text{in $\mathbb R^N$}
\end{equation}
is constant. If $N=2$ such result reduces to well know case of  entire superharmonic functions, but interestingly this property is true even if $N\geq3$, while it is false for the Laplacian. Even more it is not true also for the nonlocal counterpart of \eqref{Leq1}, in the sense that the equation
\begin{equation}\label{Leq2}
\I u=\inf_{\xi\in\mathbb S^{N-1}}\I_\xi u+\sup_{\xi\in\mathbb S^{N-1}}\I_\xi u=0\qquad\text{in $\mathbb R^N$}
\end{equation}
admits nontrivial and bounded supersolutions in any dimension $N\geq2$. To show this we borrow some arguments from \cite[Sections 3 and 4]{BGT}.
Let 
\begin{equation}\label{funz}
u(x)=\frac{1}{{(1+|x|)}^\varepsilon}
\end{equation}
where $\varepsilon>0$. Using \cite[Theorem 3.4-Remark 3.5]{BGT} we know that for $x\neq0$ 
\begin{equation}\label{Leq3}
\begin{split}
\inf_{\xi\in\mathbb S^{N-1}}\I_\xi u(x)&=\I_{x^\perp}u(x)\\
\sup_{\xi\in\mathbb S^{N-1}}\I_\xi u(x)&=\I_{\hat{x}}u(x),
\end{split}
\end{equation}
where $\hat x=\frac{x}{|x|}$ and $x^\perp$ is any unit vector orthogonal to $x$. Moreover, if we consider the operator
$$
\I^+_2u(x)=\sup\left\{\I_{\xi_1}u(x)+\I_{\xi_2}u(x)\,:\;\xi_1,\xi_2\in\mathbb S^{N-1},\;\left\langle\xi_1,\xi_2 \right\rangle=0\right\}
$$
we obtain, again by \cite[Theorem 3.4]{BGT}, that
\begin{equation}\label{Leq4}
\I^+_2u(x)=\I_{x^\perp}u(x)+\I_{\hat{x}}u(x).
\end{equation}
Hence, in view of \eqref{Leq3}-\eqref{Leq4}, we have
\begin{equation}\label{Leq5}
\I u(x)=\I^+_2u(x).
\end{equation}
Now using \eqref{Leq5} and the computations of \cite[Proposition 4.3]{BGT}, for any $x\neq0$ it holds that
\begin{equation*}
\begin{split}
\I u(x)&\leq\frac{1}{{(1+|x|)}^{\varepsilon+2s}}\left(C_s\,\text{P.V.}\,\int\limits_{-\infty}^{+\infty}\frac{{|1+\tau|}^{-\varepsilon}-1}{|\tau|^{1+2s}}\,d\tau
+C_s\int\limits_{-\infty}^{+\infty}\frac{{|1+\tau^2|}^{-\frac\varepsilon2}-1}{|\tau|^{1+2s}}\,d\tau\right)\\
&=:\frac{1}{{(1+|x|)}^{\varepsilon+2s}}\, c(\varepsilon).
\end{split}
\end{equation*}
By \cite[Proposition 3.7]{BGT}, the quantity $c(\varepsilon)$ is negative provided $\varepsilon$ is sufficiently small (depending on $s$). Lastly, since  the function $u$ cannot be touched from below at $x=0$, we conclude that $u$ is in fact a bounded viscosity supersolution of \eqref{Leq2}, as we wanted to show. It is worth to point out that $\varepsilon(s)\to0$ as $s\to 1^-$ (see \cite[Remark 3.9]{BGT}), hence the function $u$ in \eqref{funz} converges to a constant function according to the nonexistence of nontrivial supersolutions for the local equation \eqref{Leq1}. 

\medskip

\noindent
Differently from the case of supersolutions, the Liouville property is restored for solutions of \eqref{Leq2}, that we state in the following theorem. 

\begin{theorem}\label{th2}
Let $s_0$ be defined in Theorem \ref{th1}. 
If $s>s_0$ and $u\in C(\mathbb R^N)$ is a bounded viscosity solution of \eqref{Leq2}, then $u$ is constant. 
\end{theorem}
 The proof of the above theorem can be obtained by following the same arguments of \cite[Corollary 6.1]{S}, where the Liouville property is proved for a large class of integro-differential operators with kernels comparable to those of the fractional Laplacian, the key ingredients being the homogeneity of the operators considered and local H\"older estimates. The restriction on $s$ in Theorem \ref{th2} is related to Theorem \ref{th1} and any improvement concerning the range of $s$ in the estimate \eqref{HolderEstimate} provides an analogous in Theorem \ref{th2}. For the convenience of the reader we give the proof of Theorem \ref{th2}.

\begin{proof}[Proof of Theorem \ref{th2}]
Take $R>0$ and let $v(x)=u(Rx)$. Then
$$
\I v(x)=R^{2s}\I u(Rx)=0\qquad\text{in $\R^N$}.
$$
Using Theorem \ref{th1} with $\omega=B_1(0)$, $\omega'=B_2(0)$ and the fact that $\left\|v\right\|_{L^\infty(\R^N)}=\left\|u\right\|_{L^\infty(\R^N)}$, we infer that for some $\alpha\in(0,1)$
$$
\frac{|v(x)-v(y)|}{{|x-y|}^\alpha}\leq C\qquad\text{for any $x,y\in B_1(0)$,}
$$
where $C$ is a positive constant depending only $s$ and $\left\|u\right\|_{L^\infty(\R^N)}$. Rescaling back from $v$ to $u$ we then obtain
$$
\frac{|u(x)-u(y)|}{{|x-y|}^\alpha}\leq \frac{C}{{R}^\alpha}\qquad\text{for any $x,y\in B_R(0)$.}
$$
Since the right hand-side of the previous inequality tends to zero as $R\to\infty$, we conclude that $u(x)=u(y)$ for any $x,y\in\R^N$.
\end{proof}

\section*{Acknowledgements} 
\noindent
Y.S. is partially supported by NSF DMS Grant $2154219$, \lq\lq Regularity {\sl vs} singularity formation in elliptic and parabolic equations\rq\rq. I.B and G.G. are partially supported by  INdAM-GNAMPA. The authors wish to thank the Sapienza University of Rome for funding the third author's visit during the period May 1 to May 30, 2023, during which this work was started. Y.S. thanks the Department of Mathematics Guido Castelnuovo for its warm  hospitality.

\bibliographystyle{alpha}

\end{document}